\documentclass[12pt]{amsart}
\usepackage[a4paper,top=4cm,bottom=4cm,left=3.5cm,right=3.5cm]{geometry}
\usepackage{amsmath, amsthm}
\usepackage[latin1]{inputenc}
\title{Generalized Rédei rational functions and rational approximations over conics}

\newtheoremstyle{theorem}%name
  {10pt}		  % space above
  {10pt}  % space below
  {\sl}  % bofy font
  {\parindent}     % ident - empty=no indent,  \parindent= paragraph indent
  {\bf}  % thm head font
  {. }    % punctuation after thm head
  { }    % space after thm head: `` ``=normal \newline=linebreak
  {}     % thm head specification
\theoremstyle{theorem}
\newtheorem{theorem}{Theorem}

\newtheorem{proposition}[theorem]{Proposition}
\newtheorem{lemma}[theorem]{Lemma}

\newtheoremstyle{defi}%name
  {10pt}		  % space above
  {10pt}  % space below
  {\rm}  % bofy font
  {\parindent}     % ident - empty=no indent,  \parindent= paragraph indent
  {\bf}  % thm head font
  {. }    % punctuation after thm head
  { }    % space after thm head: `` ``=normal \newline=linebreak
  {}     % thm head specification
\theoremstyle{defi}
\newtheorem{definition}[theorem]{Definition}
\newtheorem{remark}[theorem]{Remark}
\newtheorem{example}[theorem]{Example}

\makeatletter\renewcommand{\section}{\@startsection{section}{1}{0mm}{\baselineskip}{\baselineskip}{\bf\normalsize\centering}}\makeatother
\begin{document}

\maketitle

\begin{center}
\author{Stefano Barbero$^1$, Umberto Cerruti$^2$, Nadir Murru$^3$\\
$^{1,2,3}$Department of Mathematics\\
University of Turin\\
8/10, Via Carlo Alberto, 10123, Torino, ITALY\\
$^1$stefano.barbero@unito.it\\
$^2$umberto.cerruti@unito.it\\
$^3$nadir.murru@unito.it}
\end{center} 
\vspace{2em}
{\bf Abstract:} In this paper we study a general class of conics starting from a quotient field. We give a group structure over these conics generalizing the construction of a group over the Pell hyperbola. Furthermore, we generalize the definition of Rédei rational functions in order to use them for evaluating powers of points over these conics. Finally, we study rational approximations of irrational numbers over conics, obtaining a new result for the approximation of quadratic irrationalities.\\
{\bf AMS Subj. Classification: 11B39, 11D09, 11J68} 

\section{Introduction}

\thispagestyle{empty}

In a previous paper \cite{bcm} we defined an operation over $\mathbb{R}^{\infty}=\mathbb{R}\cup \left\{\infty\right\}$, which  allowed us to construct a group over the Pell hyperbola 
$$H_d=\{(x,y)\in\mathbb R^2:x^2-dy^2=1\}.$$
In this paper we want to work on a more general class of conics
$$E_{\mathbb F}(h,d)=\{(x,y)\in\mathbb F^2: x^2+hxy-dy^2=1\}\,$$
over an ordinary field $\mathbb F$. In this case the conics are not only hyperbolae, but they can be ellipses or parabolae, depending on the quantity $h^2+4d$ to be positive, negative or zero respectively. Here we study some properties of such conics and, in particular, we will see how to construct a group over them, generalizing the classic construction over the Pell hyperbola. Furthermore, in \cite{bcm} we showed how Rédei rational functions \cite{Redei} can be used over the Pell hyperbola in order to find solutions of the Pell equation in a new way. Rédei rational functions are very interesting and useful in number theory, finding many applications, e.g., in the permutations of finite field, in cryptography (for these applications and a good theory about Rédei rational functions see \cite{Lidl}) or in pseudorandom sequences \cite{Topu}. Here we obtain some polynomials which generalize the Rédei rational functions and which are usable over the conics $E_{\mathbb F}(h,d)$ in order to evaluate powers of points. Finally, considering $\mathbb F=\mathbb R$, we study approximations of irrational numbers through sequences of rational numbers which lie over conics.\\
First of all we see how we can obtain the conics $E_{\mathbb F}(h,d)$ starting from a simple quotient field.
\begin{definition}
Let $\mathbb F$ be a field and $x^2-hx-d$ an irreducible polynomial over $\mathbb F[x]$. We consider the quotient field
$$\mathbb{A}=\mathbb F[x]/(x^2-hx-d)\ .$$
For any two elements $a+bx, u+vx\in\mathbb{A}$, the product naturally induced is
$$(a+bx)(u+vx)=(au+bvd)+(bu+av+bvh)x\ ,$$
while the norm, the trace and the conjugate  of an element $a+bx\in\mathbb A$, respectively, are well defined as follows:
$$N(a+bx)=a^2+hab-db^2, \quad Tr(a+bx)=2a+hb,\quad \overline{a+bx}=(a+hb)-bx\ .$$
Indeed  $$(a+bx)(\overline{a+bx})=N(a+bx)\ .$$
Consequently the inverse of an element in $\mathbb A$ is
$$ (a+bx)^{-1}=\cfrac{\overline{a+bx}}{N(a+bx)}\ . $$
\end{definition}
Now we consider the group of the unitary elements of $\mathbb A^*=\mathbb A-\{0\}$
$$U=\{a+bx\in\mathbb A^* : N(a+bx)=1\}=\{a+bx\in\mathbb A^* : a^2+hab-db^2=1\}\ .$$
So we have a natural bijection between $U$ and the set of points 
$$E=E_{\mathbb F}(h,d)=\{(x,y)\in\mathbb{F}:x^2+hxy-dy^2=1\} \ ,$$ 
which induces the commutative product $\odot_E$ over $E$
$$(x,y)\odot_E(u,v)=(xu+yvd,yu+xv+yvh)\ ,\quad \forall (x,y),(u,v)\in E\ .$$
We immediately have the following
\begin{proposition}\label{egroup}
$(E,\odot_E)$ is an abelian group with identity $(1,0)$ and the inverse of an element $(x,y)\in E$ is
$$(x,y)^{-1}=(x+hy,-y)\ .$$
\end{proposition}
We can parametrically represent the conics $E$ using 
\begin{equation} \label{line} y=\cfrac{1}{m}(x+1)\ . \end{equation}
Considering $P=\mathbb{F} \cup \{\alpha\}$, where with $\alpha$ we indicate an element not in $\mathbb F$, we directly find the bijections
\begin{equation}\label{para} \begin{cases}\epsilon:P\rightarrow E \cr \epsilon:m\mapsto \left(\cfrac{m^2+d}{m^2+hm-d}\ , \cfrac{2m+h}{m^2+hm-d}\right)\quad \forall m \in \mathbb F \cr \epsilon(\alpha)=(1,0)\ , \end{cases} \end{equation}
and
$$\begin{cases} \tau:E\rightarrow P \cr \tau:(x,y)\mapsto \cfrac{1+x}{y} \quad \forall (x,y)\in E, \quad y\not=0 \cr \tau(1,0)=\alpha \cr \tau(-1,0)=-\cfrac{h}{2} \ , \end{cases}$$
i.e. $P$ is a parametric representation of $E$. Now, using $\epsilon$ and $\tau$, we can naturally induce  a commutative product $\odot_P$ over the representation $P$
$$ \tau(s,t)\odot_P\tau(u,v)=\epsilon^{-1}((x,y)\odot_E (u,v)), \quad \forall (s,t),(u,v)\in E \ .$$
In particular, $\alpha$ becomes the identity with respect to $\odot_P$ and
\begin{equation} \label{prod} a\odot_P b=\cfrac{d+ab}{h+a+b}\quad \forall a,b\in P \ ,\ a+b\not=-h\ . \end{equation}
If $a+b=-h$, we set $a\odot_P b=\alpha$, so $a$ corresponds to the inverse of $b$ over $(P,\odot_P)$, and clearly   
\begin{proposition}
$(P,\odot_P)$ is an abelian group.
\end{proposition}
\begin{remark}
Let us consider the quotient group $B=\mathbb A^*/\mathbb F^*$, whose elements  
$$[a+bx]=\{\lambda a+\lambda bx:\lambda\in\mathbb F^*\}\ ,$$
correspond to the equivalence class  of $a+bx \in \mathbb A^*$ .
Of course if $b=0$, then $[a+bx]=[a]=[1_{\mathbb F^*}]\ $, and 
$$B=\{[a+x]:a\in \mathbb F\}\cup \{[1_{\mathbb F^*}]\}\ .$$
The product in $B$ is given by
$$[a+x][b+x]=[ab+ax+bx+x^2]=[(d+ab)+(h+a+b)x]\ ,$$
and, if $h+a+b\not=0$, then
$$[a+x][b+x]=[\cfrac{d+ab}{h+a+b}+x]\ ,$$
else
$$[a+x][b+x]=[d+ab]=[1_{\mathbb F^*}]\ ,$$
But comparing this with the product (\ref{prod}) we find an immediate isomorphism
$$\begin{cases} \phi:B\rightarrow P \cr \phi:[a+x]\mapsto a, \quad [a+x]\not=[1_{\mathbb F^*}] \cr \phi([1_{\mathbb F^*}])=\alpha \ , \end{cases}$$
which shows how $\odot_P$ can be induced in an alternative way, starting from the quotient group B.
\end{remark}

\section{Generalized Rédei rational functions}
The aim of this section is to show how Rédei rational functions can be generalized and how they are strictly related with the product $\odot_P$.
Let us recall that the $n$-th power of the  matrix
\begin{equation} \label{matrix} M=\begin{pmatrix} z+h & d \cr 1 & z  \end{pmatrix}, \quad h, d, z \in \mathbb F\ , \end{equation} 
can be determined employing two kinds of polynomials, $N_n=N_n(h,d,z)$ and $D_n=D_n(h,d,z)$, obtaining
$$M^n=\begin{pmatrix} N_n+hD_n & dD_n \cr D_n & N_n  \end{pmatrix}\ .$$
A direct calculation and an easy inductive proof can show that $N_n$ and $D_n$ are the terms of two linear recurrent sequences, as we point out in the next
\begin{remark} \label{n-d-rec}
If we indicate with $\mathcal{W}(a,b,r,k)$ the linear recurrent sequence over $\mathbb F$, with initial conditions $a, b$ and characteristic polynomial $t^2-rt+k$, then
$$N_n(h,d,z)=\mathcal{W}(1,z,2z+h,z^2+hz-d)$$
$$D_n(h,d,z)=\mathcal{W}(0,1,2z+h,z^2+hz-d)\ ,$$
\end{remark}
\noindent Moreover, in the following proposition we emphasize two important relations involving $N_n$ and $D_n$
\begin{proposition} \label{n+m} 
$$\begin{cases} N_{n+m}=N_nN_m+dD_nD_m  \cr D_{n+m}=D_nN_m+hD_nD_m+N_nD_m \ .\end{cases}$$
\end{proposition}
\begin{proof}
The proof is straightforward and only consists in comparing the resulting matrix on the right with the one on the left of the equality
$$ \begin{pmatrix} N_{n+m}+hD_{n+m} & dD_{n+m} \cr D_{n+m} & N_{n+m}  \end{pmatrix}=\begin{pmatrix} N_n+hD_n & dD_n \cr D_n & N_n  \end{pmatrix}\begin{pmatrix} N_m+hD_m & dD_m \cr D_m & N_m  \end{pmatrix}\ . $$
\end{proof}
\noindent We observe that
$$\det (M^n)=[\det(M)]^n=(z^2+hz-d)^n \ ,$$
on the other hand
$$\det(M^n)= N_n^2-hN_nD_n-dD_n^2=(z^2+hz-d)^n, $$
so, when $z^2+hz-d=1$, all points $(N_n, D_n)$ lie on $E_{\mathbb F}(h,d)\ .$\\
Finally, we can define the \emph{generalized Redèi rational functions} 
\begin{equation} \label{Q} Q_n(h,d,z)=\cfrac{N_n(h,d,z)}{D_n(h,d,z)} \quad \forall n\geq1\ . \end{equation}
Obviously, when $\mathbb F=\mathbb R$ and $h=0$, $Q_n(h,d,z)=Q_n(d,z)$, and we find the usual Rédei rational functions.
The $Q_n(h,d,z)$ have an interesting behaviour  with respect to $\odot_P$, which reveals some important aspects of these functions, in particular they can be viewed as a kind of powers. These facts are summarized in the following proposition and remarks.
\begin{proposition} For any $h,d,z \in \mathbb F$
$$Q_{n+m}(h,d,z)=Q_n(h,d,z)\odot_P Q_m(h,d,z)\ .$$
\end{proposition}
\begin{proof}
Using Proposition \ref{n+m} we have
$$\cfrac{N_n}{D_n}\odot_P\cfrac{N_m}{D_m}=\cfrac{d+\frac{N_n}{D_n}\frac{N_m}{D_m}}{h+\frac{N_n}{D_n}+\frac{N_m}{D_m}}=\cfrac{dD_nD_m+N_nN_m}{hD_nD_m+D_mN_m+D_nN_m}=\cfrac{N_{n+m}}{D_{n+m}}\ .$$
\end{proof}
\begin{remark}
Since 
$$Q_1(h,d,z)=z \ ,$$
then
$$ Q_n(h,d,z)=z^{n_{\odot_P}}=\underbrace{z\odot_P \ldots \odot_P z}_{n-times} \ , $$
i.e., the functions $Q_n$ are intimately related to $\odot_P$, essentially being powers under this product, and the multiplicative property clearly holds for these functions 
$$Q_n(h,d,Q_m(h,d,z))=(Q_m(h,d,z))^{n_{\odot_P}}=(z^{m_{\odot_P}})^{n_{\odot_P}}=z^{nm_{\odot_P}}=Q_{nm}(h,d,z)\ .$$ 
\end{remark}
\begin{remark}
If we set $Q=\{Q_n(h,d,z), \forall n\}$, then $(Q,\odot_P,\circ)$, where $\circ$ is the usual composition between functions, is a commutative ring isomorphic to $(\mathbb Z,+,\cdot)$ and
$$Q_n(h,d,\cdot):(P,\odot_P)\rightarrow(P,\odot_P) \quad \forall n\geq1\ ,$$
are morphisms, since $Q_n(h,d,z)=z^{n_{\odot_P}}\ .$
\end{remark}
\section{ Evaluating powers of points}
What we have showed in the previous section allows us to focus the attention on points of $E$ and especially on their powers. We are ready to explain how powers of points belonging to $E$ can be evaluated, using the generalized Rédei rational functions $Q_n(h,d,z)$.\\ 
Let $(x,y)$ be a point of $E$, setting $(x_n,y_n)=(x,y)^{n_{\odot_E}}$, an inductive arguments and a little bit of calculation can be easily used to prove that
\begin{equation} \label{F-G} \begin{cases} x_n=\mathcal{W}(1,x,2x+hy,1)=F_n(h,x,y)=F_n \cr y_n=\mathcal{W}(0,y,2x+hy,1)=G_n(h,x,y)=G_n\ . \end{cases}\end{equation}
Now, finding what is the image of $(x,y)^{n_{\odot_E}}$ under $\tau$ , we have the equalities
 $$\tau((x,y)^{n_{\odot_E}})=\left(\cfrac{1+x}{y}\right)^{n_{\odot_P}}=Q_n\left(h,d,\cfrac{1+x}{y}\right)\ ,$$
and
$$\tau((x,y)^{n_{\odot_E}})=\tau(x_n,y_n)=\tau(F_n,G_n)=\cfrac{1+F_n}{G_n}\ .$$
Furthermore, recalling that we are working with points $(x,y)$ of $E$, satisfying $x^2+hxy-dy^2=1$, we can eliminate the dependence from $d$ 
$$q_n(h,x,y)=Q_n\left(h,\cfrac{1-hxy-x^2}{y^2}\ ,\cfrac{1+x}{y}\right)=\cfrac{1+F_n(h,x,y)}{G_n(h,x,y)}\ .$$
Finally, by Remark \ref{n-d-rec} and (\ref{F-G}) we have
\begin{equation}\label{tau} q_n(h,x,y)=\cfrac{1+N_n(hy,x^2+hxy-1,x)}{yD_n(hy,x^2+hxy-1,x)}\ . \end{equation}
From the last equality (\ref{tau}) we discover another interesting relation, proved in the following
\begin{proposition}
Given $(x,y)\in E_{\mathbb F}(h,d)$, then
$$ q_{2n}(h,x,y)=\cfrac{F_n(h,x,y)}{G_n(h,x,y)}\ . $$
\end{proposition}
\begin{proof}
For the sake of simplicity, here we write $N_n$ and $D_n$ instead of $N_n(hy,x^2+hxy-1,x)$, $D_n(hy,x^2+hxy-1,x)$. By  (\ref{tau}), we have to show
$$\cfrac{1+N_{2n}}{yD_{2n}}=\cfrac{N_n}{yD_n}\ .$$
So we consider
$$\cfrac{1+N_{2n}}{D_{2n}}-\cfrac{N_n}{D_n}=\cfrac{D_n+D_nN_{2n}-N_nD_{2n}}{D_{2n}D_n}\ .$$
Clear consequences of Proposition \ref{n+m} are the relations
$$\begin{cases} N_{2n}=N_n^2+dD_n^2 \cr D_{2n}=2N_nD_n+hyD_n^2 \ ,  \end{cases}$$
where in this case $d=x^2+hxy-1$. Now we can easily evaluate the quantity $D_n+D_2N_{2n}-N_nD_{2n}$, finding
$$D_n+D_2N_{2n}-N_nD_{2n}=D_n+D_n(N_n^2+dD_n^2)-N_n(2N_nD_n+hyD_n^2)=$$
$$=D_n(1-N_n^2+dD_n^2-hyN_nD_n)=D_n(1-\det (M^n))=0 \ ,$$
since $\det (M^n)=x^2+hxy-d=x^2+hxy-x^2-hxy+1=1$, and $M$ is the matrix defined in (\ref{matrix}).
\end{proof}
\section{Approximations over conics}
An interesting research field involves the study of approximations of irrational numbers by sequences of rationals, which can be viewed as sequences of points over conics. In \cite{Burg} it has been proved that if a conic has a rational point, then there are irrational numbers $\beta$ such that there exists an infinite sequence of nonzero integer triples $(x_n,y_n,z_n)$ , where $\cfrac{x_n}{y_n}$ are rational approximations of $\beta$ and $\left(\cfrac{x_n}{z_n}, \cfrac{y_n}{z_n}\right)$ are rational points of the conic. Another interesting result has been proved in \cite{Els}, where rational approximations via Pythagorean triples has been studied, considering rational approximations $\cfrac{x}{y}$ of $\beta$ when  $x^2+y^2$ is a perfect square. The common point of these results is that auxiliary irrationals, depending on $\beta$, have been used and these auxiliary irrationals must have a continued fraction expansion with unbounded partial quotients (see Lemma 7 in \cite{Burg} and Theorem 1.1 in \cite{Els}). In this way, it is not possible to approximate, for example, quadratic irrationalities. Here, using powers of points studied in the previous section, we will see that we can approximate quadratic irrationalities and we do not have the problem of unbounded partial quotients.\\ 
In this section we will always consider $\mathbb F=\mathbb R$. First of all, we need the following
\begin{lemma}\label{limit}
Let $$(a_n)_{n=0}^{+\infty}=\mathcal{W}(a_0,a_1,2w,w^2-c) \quad \text{and} \quad (b_n)_{n=0}^{+\infty}=\mathcal{W}(b_0,b_1,2w,w^2-c)$$ be rational sequences, we have
$$ \lim_{n\rightarrow +\infty}\cfrac{a_n}{b_n}=\cfrac{a_1-a_0w+a_0\sqrt{c}}{b_1-b_0w+b_0\sqrt{c}}\ . $$
\end{lemma}
\begin{proof}
For the Binet formula
$$ a_n=A_1(w+\sqrt{c})^n+A_2(w-\sqrt{c})^n \quad \text{and}\quad b_n=B_1(w+\sqrt{c})^n+B_2(w-\sqrt{c})^n \ ,$$
for every $n\geq0$ and $A_1,A_2,B_1,B_2\in \mathbb C$. So
$$ \lim_{n\rightarrow +\infty}\cfrac{a_n}{b_n}=\lim_{n\rightarrow +\infty}\cfrac{A_1(w+\sqrt{c})^n+A_2(w-\sqrt{c})^n}{B_1(w+\sqrt{c})^n+B_2(w-\sqrt{c})^n}=\cfrac{A_1}{B_1}\ . $$
We can find $A_1,B_1$ simply solving the systems
$$ \begin{cases}  A_1+A_2=a_0 \cr A_1(x+\sqrt{c})+A_2(x-\sqrt{c})=a_1 \end{cases} \quad \text{and}\quad \begin{cases}  B_1+B_2=b_0 \cr B_1(x+\sqrt{c})+B_2(x-\sqrt{c})=b_1 \ .\end{cases}$$
We obtain the values
$$ A_1=-\cfrac{a_1-a_0(x-\sqrt{c})}{2\sqrt{c}} \quad \text{and} \quad B_1=-\cfrac{b_1-b_0(x-\sqrt{c})}{2\sqrt{c}}\ , $$
from which the thesis easily follows.
\end{proof}
Now we can see that powers of points, which obviously lie over the conic from Proposition \ref{egroup}, converge to a quadratic irrationality.
\begin{theorem}\label{rational}
Given a rational point $(x,y)\in E$ and $(x_n,y_n)=(x,y)^{n_{\odot_E}}$, we have
$$\lim_{n\rightarrow\infty}\cfrac{y_n}{x_n}=\cfrac{2y}{\sqrt{h^2y^2+4hxy+4x^2-4}-hy}.$$
\end{theorem}
\begin{proof}
As we have seen in the previous section 
$$(x_n)_{n=0}^{+\infty}=\mathcal W(1,x,2x+hy,1),\quad (y_n)_{n=0}^{+\infty}=\mathcal W(0,y,2x+hy,1)\ .$$
Here we will use the Lemma \ref{limit}, where the initial conditions are 
$$x_0=0,\quad x_1=y,\quad y_0=1,\quad y_1=x\ ,$$
and we have the equalities
$$2w=2x+hy,\quad w^2-c=1 \ , $$
which give
$$w=\cfrac{2x+hy}{2}\ ,\quad c=\cfrac{h^2y^2+4hxy+4x^2-4}{4}\ .$$
Thus, by the Lemma \ref{limit},
$$\lim_{n\rightarrow\infty}\cfrac{y_n}{x_n}=\cfrac{a_1-a_0w+a_0\sqrt{c}}{b_1-b_0w+b_0\sqrt{c}}=\cfrac{2y}{\sqrt{h^2y^2+4hxy+4x^2-4}-hy}\ .$$ 
\end{proof}
\begin{example}
Let us consider the conic $E=E_{\mathbb R}(-13/4,2)$ and the rational point $(4,1)$ over this conic. The powers of this point with respect to $\odot_E$ are
$$\left(18,\cfrac{19}{4}\right),\left(\cfrac{163}{2},\cfrac{345}{16}\right),\left(\cfrac{2953}{8},\cfrac{6251}{64}\right),\left(\cfrac{53499}{32},\cfrac{113249}{256}\right),\ldots \ . $$
From the last Theorem \ref{rational}, we know that
$$\left(\cfrac{1}{4},\cfrac{19}{72},\cfrac{345}{1304},\cfrac{6251}{23624},\cfrac{113249}{427992},\ldots \right)=(0.25, 0.26388, 0.26457, 0.26460, 0.26460,\ldots)$$
are rational approximations of a quadratic irrationality, which in this case is
$$\cfrac{8}{13+3\sqrt{33}}\cong 0.264605\ldots \ .$$
\end{example}
Furthermore, we can see how it is easy to construct rational approximations for every irrational number such that these approximations form points over conics. Let us consider a conic $C$ with a rational parametrization, i.e.,
$$\begin{cases} x=f(m) \cr y=g(m)\ , \end{cases}$$
for any point $(x,y)\in C$ and $f, g$ rational functions. If we take any irrational number $\beta$, we are able to construct rational approximations $\cfrac{x_n}{y_n}$ of $\beta$ such that $(x_n,y_n)\in C$. Indeed, we have only to find the irrational number $\alpha$ such that
\begin{equation} \label{alpha}\cfrac{g(\alpha)}{f(\alpha)}=\beta \end{equation}
and then to consider the continued fraction expansion of $\alpha$. We recall that a continued fraction is a representation of a real number $\alpha$ through a sequence of integers as follows:
$$\alpha=a_0+\cfrac{1}{a_1+\cfrac{1}{a_2+\cfrac{1}{a_3+\cdots}}}\ ,$$
where the integers $a_0,a_1,...$ can be evaluated with the recurrence relations 
$$\begin{cases} a_k=[\alpha_k]\cr \alpha_{k+1}=\cfrac{1}{\alpha_k-a_k} \quad \text{if} \ \alpha_k \ \text{is not an integer}  \end{cases}  \quad k=0,1,2,...$$
for $\alpha_0=\alpha$ (see \cite{Olds}). A continued fraction can be expressed in a compact way using the notation $[a_0,a_1,a_2,a_3,...]$. The finite continued fraction
\begin{equation} \label{conv} [a_0,...,a_n]=\cfrac{p_n}{q_n}\ ,\quad n=0,1,2,...\end{equation}
is a rational number and is called the $n$--th \emph{convergent} of $[a_0,a_1,a_2,a_3,...]$. 
Now, if we consider the sequences $(p_n)_{n=0}^{+\infty}$ and $(q_n)_{n=0}^{+\infty}$, coming from (\ref{conv}), the sequences $(x_n)_{n=0}^{+\infty}$ and $(y_n)_{n=0}^{+\infty}$ have general terms
\begin{equation} \label{approx}x_n=f\left(\cfrac{p_n}{q_n}\right),\quad y_n=g\left(\cfrac{p_n}{q_n}\right)\quad n=0,1,2,\ldots \ ,\end{equation}
which satisfy
$$\lim_{n\rightarrow\infty}\cfrac{y_n}{x_n}=\lim_{n\rightarrow\infty}\cfrac{g(\frac{p_n}{q_n})}{f(\frac{p_n}{q_n})}=\cfrac{g(\alpha)}{f(\alpha)}=\beta$$
and clearly $(x_n,y_n)\in C$, $\forall n\geq0$. The only conditions that we require are that equation (\ref{alpha}) has irrational solutions and the conic $C$ has a rational point (and in this case it means that it has infinite rational points). In the case of our conics $E=E_{\mathbb R}(h,d)$ we have no problems. Indeed, $(-1,0),(1,0)\in E$ and any point $(x,y)\in E$, has a parametric representation (\ref{para})
$$\begin{cases} x=f(m)=\cfrac{m^2+d}{m^2+hm-d} \cr y=g(m)=\cfrac{2m+h}{m^2+hm-d}\ . \end{cases}$$
In this case equation (\ref{alpha}) becomes
$$\cfrac{2\alpha+h}{\alpha^2+d}=\beta \ ,$$
which has solutions
$$\alpha=\cfrac{1\pm\sqrt{1+\beta^2}}{\beta}\ .$$
and $\alpha$ is always an irrational number.\\
Finally, we consider the interesting case given by $h=0, d=-1$, i.e., the conic $$E=E_{\mathbb R}(0,1)=\{(x,y)\in\mathbb R^2:x^2+y^2=1\}$$
is the unitary circle.
In this case we can construct infinite rational approximations by Pithagorean triples. Indeed, in this case equations (\ref{approx}) give
$$x_n=\cfrac{p_n^2-q_n^2}{p_n^2+q_n^2},\quad y_n=\cfrac{2p_nq_n}{p_n^2+q_n^2}\quad n=0,1,2,\ldots \ ,$$
and
\begin{equation} \label{lim-cerchio} \lim_{n\rightarrow\infty}\cfrac{y_n}{x_n}=\lim_{n\rightarrow\infty}\cfrac{2p_nq_n}{p_n^2-q_n^2}=\beta \ .\end{equation}
Since $(x_n,y_n)\in E$, $\forall n\geq0$, we trivially have
$$(p_n^2-q_n^2)^2+(2p_nq_n)^2=(p_n^2+q_n^2)^2.$$
We conclude this paper with an example of the approximation of $\pi$ over the circle.
\begin{example}
Let us consider $\beta=\pi$ the irrational number that we want to approximate over $E$. We use as auxiliary irrational 
$$\alpha=\cfrac{1+\sqrt{1+\pi^2}}{\pi}\ ,$$
which has the continued fraction expansion
$$\alpha=[1, 2, 1, 2, 1, 1, 3, 1, 1, 5,\ldots]\ .$$
The sequences $(p_n), (q_n)$ which determine the convergents are
$$(1, 3, 4, 11, 15, 26, 93, 119, 212, 1179,\ldots)$$
$$(1, 2, 3, 8, 11, 19, 68, 87, 155, 862,\ldots)\ .$$
By (\ref{lim-cerchio}) the approximations of $\pi$ are
$$\left(\cfrac{12}{5},\cfrac{24}{7},\cfrac{176}{57},\cfrac{165}{52},\cfrac{988}{315},\cfrac{12648}{4025},\cfrac{10353}{3296},\cfrac{65720}{20919},\cfrac{2032596}{646997},\ldots\right)\ ,$$
indeed they are
$$(2.4, 3.4285, 3.0877, 3.1730, 3.1365, 3.1423, 3.1410, 3.1416, 3.1415,\ldots).$$
Furthermore, the points 
$$\left(\cfrac{5}{13},\cfrac{12}{13}\right), \left(\cfrac{7}{25},\cfrac{24}{25}\right),\left(\cfrac{57}{185},\cfrac{176}{185}\right), \left(\cfrac{52}{346},\cfrac{165}{346}\right),\left(\cfrac{315}{1037},\cfrac{988}{1037}\right),\ldots $$
lie on the circle and thus we have the following Pythagorean triples
$$(5,12,13), (7,24,25), (57,176,185), (52,165,346), (315,988,1037),\ldots \ .$$
\end{example}

\end{document}